\documentclass[journal,print]{ieeecolor} 
\usepackage{amsmath,amssymb,psfrag,latexsym,times,bbm,amstext,wasysym,hyperref,url,graphicx,algorithm,algorithmic,tikz,float,multicol,caption,subcaption,balance,dsfont}
\usepackage{lcsys}
\usepackage{cite}
\usepackage{amsmath,amssymb,amsfonts,bm,bbold} 
\usepackage{mathtools}
\usepackage{algorithmic}
\usepackage{graphicx}
\usepackage{graphics}
\usepackage{textcomp}
\usepackage{dsfont}
\usepackage{epstopdf}
\usepackage{orcidlink}
\usepackage{enumerate}
\usepackage{multicol}
\usepackage{multirow}
\usepackage{empheq}
\usepackage{times}
\usepackage{subcaption}

\hypersetup{
    colorlinks=true,
    linkcolor=blue,
    filecolor=magenta,      
    urlcolor=cyan,
    citecolor=blue,
    pdfpagemode=FullScreen,
    }

\newtheorem{thm}{Theorem}

\newtheorem{prop}{Proposition}

\newtheorem{problem}{Problem}

\newcommand{\tr}{{\rm tr}}

\newtheorem{taggedproblemx}{Problem}
\newenvironment{taggedproblem}[1]
 {\renewcommand\thetaggedproblemx{#1}\taggedproblemx}
 {\endtaggedproblemx}

\def\BibTeX{{\rm B\kern-.05em{\sc i\kern-.025em b}\kern-.08em
    T\kern-.1667em\lower.7ex\hbox{E}\kern-.125emX}}
\begin{document}
\title{
Collective Steering: Tracer-Informed Dynamics
}

\author{Asmaa Eldesoukey, Mahmoud Abdelgalil, and Tryphon T. Georgiou
\thanks{Asmaa Eldesoukey and Tryphon T. Georgiou are with the Department of Mechanical and Aerospace Engineering, University of California, Irvine, Irvine, CA, USA, (emails: aeldesou@uci.edu,tryphon@uci.edu).}
\thanks{Mahmoud Abdelgalil is with the Department of Electrical and Computer Engineering, University of California, San Diego, La Jolla, CA, USA, (email:mabdelgalil@ucsd.edu).}
\thanks{The research was supported in part by the NSF under ECCS-2347357, AFOSR under FA9550-24-1-0278, and ARO under W911NF-22-1-0292.}}
\maketitle
\thispagestyle{plain}
\pagestyle{plain}

\begin{abstract}
We consider control and inference problems where
control protocols and internal dynamics are informed by two types of constraints.
Our data consist of i) statistics on the ensemble and ii) trajectories or final dispositions of selected tracer particles embedded in the flow. Our aim is i') to specify a control protocol, realizing a flow that meets such constraints or ii') to recover the internal dynamics that are consistent with such a data set.
We analyze these problems in the setting of linear flows and Gaussian distributions. The control cost is taken to be a suitable action integral constrained by either the trajectories of tracer particles or their terminal placements. 
%
\end{abstract}
%
%
\section{Introduction}

The theme of this work is the problem of steering a collection of particles/agents that obey a universal time-varying feedback law; the law represents a common control protocol broadcast to all agents in the collective. In this paper, we are interested in both control and inference problems. The data for these problems concern, on one hand, the collective in the form of marginal distribution constraints, and on the other, the locations or trajectories of specific {\em tracer} particles, namely tagged agents in the flow. Whether we consider inference or control, the tracer particles are to reveal internal dynamics or constrain a sought-after control protocol.

The practical motivation for this class of problems stems from a range of applications in engineering, physics, and biology. 
To mention a few, this is the case in particle systems steered by externally regulated control potentials \cite{schneider2019optimal,fu2021maximal}. Similarly, in swarm robotics, exogenous control inputs are used to steer the swarm to desired configurations \cite{peyer2013bio,shahrokhi2017steering}. On the flip side, landmarks in the field of view of a collective may serve as tracer particles' constraints to guide the motion of the ensemble traversing a known terrain. Indeed, the use of tracers in medical imaging should not be overlooked \cite{jung2020whole,schwenck2023advances}.
Lastly, tracer particles can also be used in fluids to reveal the internal dynamics of the flow.

The present work aims to develop this circle of ideas in the most basic setting of linear flows and Gaussian distributions. In this, we consider
linear time-varying feedback control laws. Control costs penalize both the kinetic energy of the ensemble as well as the control gain. The former is physically motivated and underlies the theory of $L_2$-optimal mass transport \cite{benamou2000computational,villani2003topics,chen2016optimal,peyre2019computational,takatsu2011wasserstein,modin2016geometry}. The latter quantifies the complexity of control implementation as was proposed and eloquently explained by R.W.\ Brockett \cite{brockett1997minimum}.

The average kinetic energy has been a typical cost to transportation problems due to its convexity. 
Notably, the interpolation of measures given in \cite{mccann1997convexity}, and known as McCann's interpolation, coincides with the solution to the $L_2$-optimal transport problem. Moreover, in the context of control systems, \cite{chen2016relation} and \cite{chen2018optimal} considered minimizing the control energy to steer an ensemble that obeys linear dynamics between the endpoint distributions, thereby solving an $L_2$-optimal transport problem with prior dynamics. Further, \cite{chen2018state} extended the problem to tracking a sequence of distributions based on output measurements, instead of assuming full-state knowledge. Here, we also refer to the review article \cite{haasler2021control} and the references within.

 In the present work, given the data on the collective and tracers, we shall be concerned with identifying a suitable {\em pair} of a feedback gain and a state transition matrix that reconciles with the data. While seeking feedback gain is ubiquitous in control problems, identifying associated state transitions is novel and, in fact, essential in our exposition. In that, the purpose of obtaining the state-transitions is to recover the internal degrees of freedom which can not be inferred from solving the $L_2$-optimal transport problem; under the assumptions of linear flows and Gaussian distributions, the geometric study \cite{modin2016geometry} (see also \cite{takatsu2010wasserstein} and \cite{takatsu2011wasserstein}) explains that the $L_2$-optimal transport is a geodesic problem on the general linear group. The geodesics constituting the solution to optimal transport solely utilize the horizontal distribution, which is identified with gradient vector fields. Therefore, the problem is reduced to a geodesic problem on the space of Gaussian densities, where geodesics remarkably coincide with those of McCann's interpolation. This is to say, solving the $L_2$-optimal transport problem can only provide information about the underlying flow based on the symmetric part of the state transitions. As such, we are motivated here to obtain state transitions in full, utilizing tracer-based knowledge.

In the case at hand, when the feedback gain and the state transitions are decision variables, the dynamics turn out to be bilinear. Due to this fact, controllability issues are subtle and often overlooked.
For instance, while it was widely accepted that, in this basic setting, the Liouville equation is controllable under standard conditions \cite{brockett2007optimal}, a proof of this fact was not available until recently \cite{chen2015optimal} (see also \cite{abdelgalil2025collective}).
At the heart of such a technical issue is the fact that solutions to optimal control problems for linear systems may not always be expressible in linear feedback form due to topological obstructions \cite[Section IV]{abdelgalil2025collective}. Thus, besides raising attention to this class of problems where tracer particles may inform control and inference, it is worthwhile to shed light on the delicacy in establishing the existence of solutions to such problems, possibly under suitable regularization of the cost.

Below, we introduce notation and discuss the general setting in Section \ref{sec:setting}. In Section \ref{sec:problem1}, we introduce and solve the first of our problems: the problem of collective steering with the information on the running distribution of the collective in conjunction with the endpoint positions of tracers. In Section \ref{sec:problem 2}, we formulate and solve a dual problem involving information on the collective's endpoint distributions and trajectories of the tracers. Before concluding, we present an academic example for each problem in Section \ref{sec:examples}.


\section{Problem setting}\label{sec:setting}
Throughout, we consider a collection of particles in transport over the time window $t \in [0,1]$ that are subject to the linear dynamics:
\begin{align}
    \dot{X}_t &= u_t(X_t)  = K_t X_t \label{eq:lineardyn}.
\end{align}
We assume that $X_0 \sim \mathcal N_n (0,\Sigma_0)$ with $\Sigma_0\in \mathrm{Sym}^+(n)$, where $\mathcal N_n (0,\Sigma_0)$ denotes the $n$-dimensional Gaussian distribution with zero mean and covariance $\Sigma_0$, and
\begin{align*}
    {\rm Sym}^+(n) := \{ \Sigma \in \mathbb R^{n \times n} ~|~ \Sigma \succ0, \Sigma = \Sigma^\top\},
\end{align*}
i.e., the space of  $n \times n$ positive-definite symmetric matrices.
If ${\rm GL}(n)$ denotes the real general linear group of degree $n$, i.e., the multiplicative group of $n\times n$ invertible matrices, then the flow of \eqref{eq:lineardyn} takes the form
\begin{align}\label{eq:linearflow}
    X_t=\Phi_t X_0,
\end{align}
where the state-transition matrix $\Phi_t \in {\rm GL}(n)$ satisfies
\begin{align}\label{eq:state-transition}
    \dot \Phi_t = K_t \Phi_t,  \quad \mbox{ and } \Phi_0 = I,
\end{align}
for $I$ denoting the identity matrix. As is well-known, the space of Gaussian distributions is invariant under the action of linear flows, i.e., flows of the form \eqref{eq:linearflow}, via the pushforward of measures. Then, with direct computation, one can verify that $X_t \sim \mathcal N_n(0, \Phi_t\Sigma_0\Phi_t^\top)$. 

As outlined in the introduction, the data of the problems we consider in this paper belong to one of two types: i) marginal distributional constraints and ii) trajectories/displacement constraints of \emph{tracer} particles. Expressed in terms of the state transition matrix $\Phi_t$ defining the flow of \eqref{eq:lineardyn}, the first of these takes the form
\begin{align}\label{eq:covariance-constraint}
    \Sigma_t = \Phi_t \Sigma_0 \Phi_t^\top, \quad \mbox{ with } t\in\mathcal{T}_1\subseteq[0,1],
\end{align}
where $\Sigma_t\in \mathrm{Sym}^+(n)$. 
We will only be interested in the cases where $\mathcal{T}_1=\{0,1\}$ or $\mathcal{T}_1=[0,1]$. In the latter case, we shall always assume that the curve $(\Sigma_t)_{t\in[0,1]}$ is continuously differentiable. 

The second type of constraints, ii), can be succinctly explained as follows. Let $y_t^{(i)}$ be the position of the $i$-th tracer particle at the time instant $t\in\mathcal{T}_2\subseteq[0,1]$ for $i \in \{1, \ldots m\}$ and let
\begin{align*}
    Y_t = \begin{bmatrix}
        y_t^{(1)} & \ldots & y_t^{(m)} 
    \end{bmatrix} \in \mathbb R^{n \times m}
\end{align*}
be the matrix comprising instantaneous information on the tracer positions. Then, the constraint ii) is of the form
\begin{align} \label{eq:Yt-constraint}
    Y_t=\Phi_t Y_0, \quad \mbox{ with } t\in\mathcal{T}_2\subseteq[0,1].
\end{align}
Once again, we will only be interested in the cases where $\mathcal{T}_2=\{0,1\}$ or $\mathcal{T}_2=[0,1]$, and in the latter case, we shall always assume that the curve $(Y_t)_{t \in[0,1]}$ is continuously differentiable and of full column rank.

We note that in the limiting case $m=n$ and by \eqref{eq:Yt-constraint}, the matrix $\Phi_t = Y_t^{-1} Y_0$, i.e., the state transitions are fully determined from information on the tracer trajectories. In what follows, we shall assume that $m<n$, and we will be concerned with determining the state transitions that reconcile with given information on the flow via optimization problems. Specifically, we will be concerned with minimizing a cost functional over a to-be-characterized admissible set of pairs $(K_t,\Phi_t)$. We shall refer to the global minimum in the admissible set as the {\em absolute} minimum, following the terminology of \cite{cesari1983optimization}.

Before concluding this section, we introduce the cost functionals we will be dealing with. The first control cost we consider is the integral of the (expected) kinetic energy:
\begin{align}\label{eq:cost-kinetic-energy-X}
    J_{\rm KE} := \frac{1}{2}\int_0^1 \mathbb{E}[\|\dot{X}_t\|^2]\,\mathrm{d}t,
\end{align}
where $\mathbb{E}[\cdot]$ denotes the expectation operator (recall that $X_0\sim\mathcal{N}_n(0,\Sigma_0)$, i.e., $X_0$ is a Gaussian random variable). The functional in \eqref{eq:cost-kinetic-energy-X} is a typical cost in optimal mass transport and quantifies the \emph{control energy}, see, e.g., \cite{chen2016optimal} and \cite{modin2016geometry}. Substituting \eqref{eq:lineardyn} and \eqref{eq:linearflow} into \eqref{eq:cost-kinetic-energy-X} while utilizing the properties of the expectation operator and the fact that $\mathbb{E}[X_0X_0^\top]=\Sigma_0$, we obtain that  \eqref{eq:cost-kinetic-energy-X} can be rewritten as
\begin{align}\label{eq:cost-kinetic-energy-Phi}
  J_{\rm KE} = \frac{1}{2}  \int_0^1 \tr(K_t\Phi_t\Sigma_0\Phi_t^\top K_t^\top)\,\mathrm{d}t  .
\end{align}
Notice that, in terms of $K_t$ and $\Phi_t$, the cost functional in \eqref{eq:cost-kinetic-energy-Phi} is not jointly convex, and so establishing the existence of solutions of optimal control problems with the cost in \eqref{eq:cost-kinetic-energy-Phi} is non-trivial.
Indeed, for some terminal specifications, minimizing the cost in \eqref{eq:cost-kinetic-energy-Phi} leads to a singular state transition $\Phi_t$ and, consequently, a non-integrable $K_t$ \cite[Example 1]{abdelgalil2025collective}.
The second cost functional we consider is expressed as
\begin{align}\label{eq:attention-cost-X}
  J_{\rm A} :=  \frac{1}{2} \int_0^1 \mathbb{E}[\|\nabla u(X_t)\|^2]\,\mathrm{d}t,
\end{align}
which is a modification of the attention functional introduced by R.W. Brockett in \cite{brockett1997minimum} and represents a \emph{weighted complexity} of implementing the control law $u_t$ in \eqref{eq:lineardyn} (the weighting in the form of the expectation is introduced to ensure a finite integral). From \eqref{eq:lineardyn} and \eqref{eq:linearflow}, we see that \eqref{eq:attention-cost-X} is equivalent to
\begin{align}\label{eq:attention-cost-Phi}
   J_{\rm A} = \frac{1}{2}  \int_0^1 \tr(K_t K_t^\top)\,\mathrm{d}t.
\end{align}
Note that, unlike the functional in \eqref{eq:cost-kinetic-energy-Phi}, the cost functional in \eqref{eq:attention-cost-Phi} is convex in $K_t$ and is independent of $\Phi_t$. It turns out that utilizing the functional in \eqref{eq:attention-cost-Phi} as a \emph{regularization} term for the cost functional in \eqref{eq:cost-kinetic-energy-Phi} is instrumental in establishing the existence of solutions to one of the problems we consider.

\section{Main Contributions}
\subsection{Internal Dynamics from Tracer Displacements}\label{sec:problem1}
Here, we tackle the first of the problems we propose in this paper. Throughout this subsection, we 
assume knowledge about the ensemble of particles in the form of the marginal distributional constraints \eqref{eq:covariance-constraint} with $\mathcal{T}_1=[0,1]$ and the endpoint conditions on the tracers \eqref{eq:Yt-constraint} with $\mathcal{T}_2=\{0,1\}$. Our task is to identify the dynamics in \eqref{eq:state-transition}, affected by a suitable control vector field $u_t$  as in \eqref{eq:lineardyn}, that realizes the given data about the flow. We take the functional in \eqref{eq:cost-kinetic-energy-Phi} as our control cost. The problem we investigate can then be formulated as follows.
\begin{problem}\label{problem:knowledge-on-ens}
    Given $(\Sigma_t)_{t \in \mathcal{T}_1} \in {\rm Sym}^+(n)$ with $\mathcal{T}_1=[0,1]$ and $(Y_t)_{t  \in \mathcal{T}_2} \in \mathbb R^{n \times m}$ with $\mathcal{T}_2=\{0,1\}$, find
    \begin{align*}
         (K_t^\star, \Phi_t^\star) := \arg \min_{(K_t, \Phi_t)} J_{\rm KE}
    \end{align*}
   subject to the constraints \eqref{eq:state-transition}, \eqref{eq:covariance-constraint}, and \eqref{eq:Yt-constraint}.
\end{problem}

Let us define what shall be viewed as an admissible solution to the previous problem. We say a curve $(K_t,\Phi_t)_{t\in[0,1]}$ is \emph{admissible} for Problem~\ref{problem:knowledge-on-ens} if $(K_t)_{t \in[0,1]}$ is square integrable and $(\Phi_t)_{t \in[0,1]}$ is absolutely continuous, and are such that the differential constraint \eqref{eq:state-transition} is satisfied (whenever $\dot{\Phi}_t$ exists) as well as constraints \eqref{eq:covariance-constraint} and \eqref{eq:Yt-constraint}, and the value of the functional \eqref{eq:cost-kinetic-energy-Phi} is finite. The following proposition provides a more explicit characterization of the class of admissible curves.

\begin{prop}\label{prop:admissible-curves-problem-1}
    A curve $(K_t,\Phi_t)_{t\in[0,1]}$ is an admissible curve for Problem~\ref{problem:knowledge-on-ens} if and only if the following four conditions are satisfied. First,
    \begin{align}\label{eq:push-lin}
        \Phi_1\Sigma_0\Phi_1^\top=\Sigma_1,
    \end{align}
    second, the constraint \eqref{eq:Yt-constraint} is satisfied, third, there exists an almost everywhere unique square-integrable curve $(\Omega_t)_{t\in[0,1]}\in\mathbb{R}^{n\times n}$ such that
    \begin{align}\label{eq:solution-form-problem-1}
        K_t=\frac{1}{2}\dot{\Sigma}_t\Sigma_t^{-1}+ \Omega_t\Sigma_t^{-1},\quad \mbox{ with } \Omega_t^\top=-\Omega_t,
    \end{align}
    for almost all $t\in[0,1]$, and fourth, the differential constraint \eqref{eq:state-transition} is satisfied (whenever $\dot{\Phi}_t$ exists).
\end{prop}

\begin{proof}
    Let $(K_t,\Phi_t)_{t\in[0,1]}$ be an admissible curve for Problem~\ref{problem:knowledge-on-ens}. Since $\mathcal{T}_1=[0,1]$ and $(\Sigma_t)_{t\in\mathcal{T}_1}$ is continuously differentiable by assumption, direct differentiation of \eqref{eq:covariance-constraint} followed by substitution of \eqref{eq:state-transition} shows that $(K_t)_{t \in[0,1]}$ must satisfy the Lyapunov equation:
    \begin{align}\label{eq:lyap}
        \dot{\Sigma}_t = K_t \Sigma_t + \Sigma_t K_t^\top,
    \end{align}
    for almost all $t\in[0,1]$.  On the other hand, since $\Sigma_t\in\mathrm{Sym}^+(n)$, it follows from basic linear algebra that the constraint \eqref{eq:lyap} is satisfied for all $t\in\mathcal{T}_1=[0,1]$ if and only if $K_t$ is given by \eqref{eq:solution-form-problem-1}, where the square integrable-curve $(\Omega_t)_{t\in[0,1]}\in\mathbb{R}^{n\times n}$ is almost everywhere uniquely defined. Finally, since $1\in\mathcal{T}_1=[0,1]$, the curve $(\Phi_t)_{t \in[0,1]}$ must be such that $\Phi_1$ satisfies \eqref{eq:push-lin}. The converse implication follows by direct integration of the differential constraint \eqref{eq:state-transition} with $(K_t)_{t \in[0,1]}$ defined by \eqref{eq:solution-form-problem-1}.
\end{proof}

Next, we observe that if $(K_t,\Phi_t)_{t\in[0,1]}$ is any admissible curve for Problem~\ref{problem:knowledge-on-ens}, the functional in \eqref{eq:cost-kinetic-energy-Phi} is equivalent to
\begin{align*}
  \frac{1}{2}  \int_0^1 \tr(K_t \Sigma_t K_t^\top)\,\mathrm{d}t,
\end{align*}
where we have substituted the constraint \eqref{eq:covariance-constraint} into \eqref{eq:cost-kinetic-energy-Phi}. The advantage of the latter representation of the functional is that it is convex in the free parameter $\Omega_t$ (after another substitution from \eqref{eq:solution-form-problem-1}) and is independent of $\Phi_t$. 
The preceding discussion suggests the reformulation of Problem~\ref{problem:knowledge-on-ens} as follows.
\begin{taggedproblem}{\ref{problem:knowledge-on-ens}$^\prime$}\label{problem:knowledge-on-ens-reformulation}
    Given $(\Sigma_t)_{t \in \mathcal{T}_1} \in {\rm Sym}^+(n)$ with $\mathcal{T}_1=[0,1]$ and $(Y_t)_{t  \in \mathcal{T}_2} \in \mathbb R^{n \times m}$ with $\mathcal{T}_2=\{0,1\}$, find 
    \begin{align*}
     (K_t^\star, \Phi_t^\star) := \arg \min_{(K_t, \Phi_t)} \frac{1}{2} \int_0^1 \tr(K_t \Sigma_t K_t^\top)\,\mathrm{d}t 
    \end{align*}
    subject to the constraints \eqref{eq:state-transition}, \eqref{eq:Yt-constraint}, \eqref{eq:push-lin}, and \eqref{eq:solution-form-problem-1}.
\end{taggedproblem}
The next proposition, the proof of which is a direct consequence of Proposition~\ref{prop:admissible-curves-problem-1} and the discussion preceding the statement of Problem~\ref{problem:knowledge-on-ens-reformulation}, formulates the equivalence between Problem~\ref{problem:knowledge-on-ens} and Problem~\ref{problem:knowledge-on-ens-reformulation}.   
\begin{prop}
    An admissible curve $(K_t,\Phi_t)_{t\in[0,1]}$ solves Problem~\ref{problem:knowledge-on-ens} if and only if it solves Problem~\ref{problem:knowledge-on-ens-reformulation}.
\end{prop}

Before we proceed with deriving the necessary conditions of optimality, we state in passing a theorem on the existence of solutions to Problem~\ref{problem:knowledge-on-ens-reformulation} and, by extension, Problem~\ref{problem:knowledge-on-ens}.

\begin{thm}\label{thm:existence-of-solutions-problem-1}
    Suppose at least one admissible curve exists for Problem~\ref{problem:knowledge-on-ens}. Then, there exists an absolute minimizer $(K_t^\star,\Phi_t^\star)_{t\in[0,1]}$ for Problem~\ref{problem:knowledge-on-ens}.
\end{thm}

The theorem implies that Problem~\ref{problem:knowledge-on-ens} has an absolute minimizer, which, however, may not be unique. The proof, due to its technical nature, is omitted. Briefly, it consists of verifying that the conditions of \cite[Theorem 11.4.vi]{cesari1983optimization} hold.
%
%
Thus, in light of Theorem~\ref{thm:existence-of-solutions-problem-1}, we can now meaningfully seek a minimizer by solving the necessary conditions of optimality obtained from the classical theory of optimal control \cite{kirk2004optimal} as formulated next.

\begin{prop} \label{prop:proptracer}
    Let $(K_t^\star,\Phi_t^\star)_{t\in[0,1]}$ be an absolute minimizer of Problem~\ref{problem:knowledge-on-ens}, or equivalently Problem~\ref{problem:knowledge-on-ens-reformulation}. Then, 
    \begin{subequations}
    \begin{align}\label{eq:neccPhi}
        \dot{\Phi}_t^\star &=  K_t^\star \Phi_t^\star, 
        \end{align}
        with $\Phi_0^\star = I$, $K_t^\star$ coincides with \eqref{eq:solution-form-problem-1} and $\Omega_t$ satisfies
        \begin{align}
        \begin{split}
           \Omega_t \Sigma_t+  \Sigma_t \Omega_t &=   \ \big( \Phi_t^\star (P_t^\star)^\top + \frac{1}{2}\dot{\Sigma}_t\big)\Sigma_t \\
        &   -  \Sigma_t\big( P_t^\star (\Phi_t^\star)^\top + \frac{1}{2} \dot{\Sigma}_t \big), \end{split}  \label{eq:lyap-omega} \\
     \hspace{-30pt}   \text{where } \hspace{20pt}   \dot{P_t}^\star& =  -({K_t}^\star)^\top P_t^\star.         \label{eq:neccP}             
        \end{align}
        Moreover, if $\mathcal S\subset\mathrm{GL}(n)$ is the surface defined by
        \begin{align} \label{eq:surfnecc}
            \mathcal{S}:=\{\Phi\in\mathrm{GL}(n)~|~\Phi\Sigma_0\Phi^\top=\Sigma_1,\,\Phi Y_0= Y_1\},
            \end{align}
      \end{subequations}
        then \eqref{eq:neccPhi} and \eqref{eq:neccP} are subject to the mixed boundary conditions $\Phi_1^\star  \in \mathcal S$ and ${P_1} \perp T_{\Phi_1^\star}\mathcal{S}$, respectively.        
\end{prop}

\begin{proof}
   From Proposition~\ref{prop:admissible-curves-problem-1}, the control Hamiltonian corresponding to Problem~\ref{problem:knowledge-on-ens-reformulation} is
   \begin{align*}
       \mathcal H(K_t, \Phi_t,P_t) &=  \frac{1}{2} \tr\big( K_t  \Sigma_t K_t^\top) + \tr(P_t^\top K_t \Phi_t  \big),
   \end{align*}
   for $K_t$ satisfying \eqref{eq:solution-form-problem-1} and where $P_t$ is an $n \times n$ Lagrange multiplier (costate) matrix that enforces \eqref{eq:state-transition}. 
    The first order necessary conditions in $P_t$ and  $\Phi_t$ are such that $\dot{\Phi}_t^\star  = \partial \mathcal H/\partial P_t $ and $\dot{P}_t^\star  = -\partial \mathcal H/\partial \Phi_t $, respectively. These two conditions yield Equations \eqref{eq:neccPhi} and \eqref{eq:neccP}, in that order.

    Using \eqref{eq:solution-form-problem-1}, we rewrite the previous Hamiltonian as
    \begin{align*}
         \mathcal H(\Omega_t, \Phi_t,P_t) &=  \frac{1}{2} \tr\big( \big(\frac{1}{2}\dot{\Sigma}_t+\Omega_t \big)\Sigma_t^{-1}  \big(\frac{1}{2}\dot{\Sigma}_t+\Omega_t\big)^\top\big) \\
        &+  \tr \big(P_t^\top\big(\frac{1}{2}\dot{\Sigma}_t+\Omega_t\big) \Sigma_t^{-1} \Phi_t  \big),
    \end{align*}
    bearing in mind that $\Omega_t$ is skew-symmetric. 
    In this case, Pontryagin's maximum principle is such that $\partial \mathcal H /\partial \Omega_t  =0$, from which
    the Lyapunov equation \eqref{eq:lyap-omega} can be obtained.

    The terminal condition \eqref{eq:push-lin} along with tracers' condition \eqref{eq:Yt-constraint} define the surface $\mathcal S$ in \eqref{eq:surfnecc}. Thus, it is necessary that $\Phi_1^\star  \in \mathcal S$. According to \cite[Section 5.1]{kirk2004optimal}, any admissible variation for the optimal value $\Phi_1^\star$, $\delta \Phi_1^\star$, has to be tangent to $\mathcal S$, i.e., $\delta\Phi_1^\star \in T_{\Phi_1^\star}\mathcal{S}$. We note,
    \begin{align*}
        \tr(({P_1^\star})^\top \delta \Phi_1^\star)=0,
    \end{align*}
    for any admissible variation $\delta \Phi_1^\star$, cf. \cite[Equation 5.1-18]{kirk2004optimal}, which necessitates that ${P_1^\star} \perp T_{\Phi_1^\star}\mathcal{S}$.    
\end{proof}

Any solution to Problem~\ref{problem:knowledge-on-ens} must satisfy the necessary conditions of optimality stated in Proposition~\ref{prop:proptracer}, constituting a nonlinear two-point boundary value problem. A solution to this problem, which is guaranteed to exist as a consequence of Theorem~\ref{thm:existence-of-solutions-problem-1}, provides a candidate flow that reconciles the data of Problem~\ref{problem:knowledge-on-ens} with the particle dynamics \eqref{eq:lineardyn} using a suitable linear feedback law. Numerical results, in addition to visualization of the flow, are presented in Section \ref{sec:example-1} to illustrate the nature of the solutions to this problem.

\subsection{Internal Dynamics from Tracer Trajectories}\label{sec:problem 2}

This subsection deals with a counterpart of Problem~\ref{problem:knowledge-on-ens}. In it, the data of the problem we consider here is in the form of the distributional constraints \eqref{eq:covariance-constraint} with $\mathcal{T}_1=\{0,1\}$ and the full trajectories of the tracers \eqref{eq:Yt-constraint} with $\mathcal{T}_2=[0,1]$. Our task is once again to identify the dynamics in \eqref{eq:state-transition}. If \eqref{eq:cost-kinetic-energy-Phi} is taken as the cost of transport, then a candidate problem formulation can be stated as follows. 
   \begin{problem}\label{problem:singular}
  Given $(\Sigma_t)_{t \in \mathcal{T}_1} \in {\rm Sym}^+(n)$ with $\mathcal{T}_1=\{0,1\}$ and $(Y_t)_{t  \in \mathcal{T}_2} \in \mathbb R^{n \times m}$ with $\mathcal{T}_2=[0,1]$, find
    \begin{align*}
         (K_t^\star, \Phi_t^\star) := \arg \min_{(K_t, \Phi_t)} J_{\rm KE}
    \end{align*}
    subject to the constraints \eqref{eq:state-transition}, \eqref{eq:covariance-constraint}, and \eqref{eq:Yt-constraint}.
    \end{problem}

However, this candidate formulation may not admit a solution, in general; there are cases where $\Phi_t $ becomes singular at some point in time, see \cite[Example 1]{abdelgalil2025collective}. To avert this issue, we modify the cost in Problem~\ref{problem:singular} by augmenting it with the ``weighted attention" term in \eqref{eq:attention-cost-Phi}, which penalizes the $L_2$-norm of the gain $K_t$. Such an augmentation traces back to R.W. Brockett \cite{brockett1997minimum}, according to whom,
    \begin{quote}
  \textit{``... the easiest control law to
implement is a constant input. Anything else requires some attention. The more frequently the control changes, the more effort it takes to implement it.
[...], it can be argued that the cost of implementation is linked to the rate at which the control changes with changing values of $x$ [state] and $t$ [time]." }
    \end{quote}

        For convenience, we restrict ourselves to a spatial attention term as in \eqref{eq:attention-cost-X} or \eqref{eq:attention-cost-Phi}. Thereby, we state the adaptation of Problem~\ref{problem:singular} as follows.    
\begin{problem} \label{problem:dual}
    Given $(\Sigma_t)_{t \in \mathcal{T}_1} \in {\rm Sym}^+(n)$ with $\mathcal{T}_1=\{0,1\}$ and $(Y_t)_{t  \in \mathcal{T}_2} \in \mathbb R^{n \times m}$ with $\mathcal{T}_2=[0,1]$, find 
           \begin{align}\label{eq:cost-with-attention}
 (K_t^\star, \Phi_t^\star): = \arg \min_{(K_t, \Phi_t)} J_{\rm KE} + \varepsilon J_{\rm A},
  \end{align}
    for $\varepsilon >0$ and subject to the constraints \eqref{eq:state-transition}, \eqref{eq:covariance-constraint}, and \eqref{eq:Yt-constraint}. 
\end{problem}

A curve $(K_t,\Phi_t)_{t\in[0,1]}$ is said to be \emph{admissible} for Problem~\ref{problem:dual} if $(K_t)_{t \in[0,1]}$ is square integrable and $(\Phi_t)_{t \in[0,1]}$ is absolutely continuous, and are such that \eqref{eq:state-transition} is satisfied (whenever $\dot{\Phi}_t$ exists) as well as constraints \eqref{eq:covariance-constraint} and \eqref{eq:Yt-constraint}, and the value of the functional \eqref{eq:cost-with-attention} is finite. The following proposition provides a more explicit characterization of the class of admissible curves.

\begin{prop}\label{prop:admissible-curves-problem-2}
    Fix $N_0\in\mathbb{R}^{(n-m)\times n}$ such that
    \begin{align}\label{eq:Mat-N0}
        N_0 Y_0 = 0, ~\mbox{ and } \ \mathrm{rank}([Y_0, \,N_0^\top])= n.
    \end{align}
    Then, a curve $(K_t,\Phi_t)_{t\in[0,1]}$ is an admissible curve for Problem~\ref{problem:dual} if and only if the following four conditions hold. First, the constraint \eqref{eq:covariance-constraint} is satisfied, second,
    \begin{align}\label{eq:tracer-displacement}
        Y_1=\Phi_1 Y_0,
    \end{align}
    third, there exists an almost everywhere unique square integrable curve $(R_t)_{t\in[0,1]}\in\mathbb{R}^{n\times (n-m)}$ such that
    \begin{subequations}\label{eq:solution-form-problem-2}
        \begin{align}
            K_t &= M_t + R_t N_t, \\
           \text{ with } \ M_t &= \dot{Y}_t (Y_t^\top Y_t)^{-1} Y_t^\top, \label{eq:solution-form-problem-2-Mt} \\
           \text{ and } \ \dot{N}_t & = - N_t M_t, \label{eq:dynamics-Nt}
        \end{align}
    for almost all $t\in[0,1]$, and fourth, the constraint \eqref{eq:state-transition} is satisfied (whenever $\dot{\Phi}_t$ exists).
    \end{subequations}
\end{prop}

\begin{proof}
    Let $(K_t,\Phi_t)_{t\in[0,1]}$ be an admissible curve for Problem~\ref{problem:dual}.  Since $\mathcal{T}_2=[0,1]$ and $(Y_t)_{t \in\mathcal{T}_2}$ is continuously differentiable by assumption, direct differentiation of \eqref{eq:Yt-constraint} followed by substitution of \eqref{eq:state-transition} shows that $(K_t)_{t \in[0,1]}$ must satisfy
    \begin{align}\label{eq:linear-eqn}
        \dot{Y}_t = K_t Y_t,
    \end{align}
    for almost all $t\in[0,1]$. It follows that
    \begin{align}    \label{eq:null-Zt}
    K_t = M_t + Z_t,
    \end{align}
    where $(Z_t)_{t\in[0,1]}\in\mathbb{R}^{n\times n}$ is any square integrable curve satisfying $Z_t Y_t= 0$, i.e., the column space of $Y_t$ is a subspace of the null space of $Z_t$.
    
    To proceed, let $N_0\in \mathbb{R}^{(n-m)\times n}$ be any full row-rank matrix such that $N_0 Y_0 = 0$ and define the curve $(N_t)_{t\in[0,1]}$ as the solution to
    \begin{align}\label{eq:ivp-nullspace-basis}
        \dot{N}_t & := - N_t M_t,
    \end{align}
    from the initial condition $N_0$, where $M_t$ is as in \eqref{eq:solution-form-problem-2-Mt}. We claim that $N_{t}$ is of full row rank, and the null space of $N_t$ equals the column space of $Y_t$ for almost all $t\in[0,1]$. To see this, note that \eqref{eq:ivp-nullspace-basis} is a linear time-varying system of equations, and so the solution to \eqref{eq:ivp-nullspace-basis} from the initial condition $N_0$ is uniquely given by
    \begin{align*}
        N_t&= N_0\Psi_t,
    \end{align*}
    where $\Psi_t$ is the fundamental solution matrix associated with the linear system \eqref{eq:ivp-nullspace-basis}. Since $\Psi_t$ is invertible for every $t$, it follows that $N_{t}$ is full row rank if and only if $N_0$ is full row rank. On the other hand, direct differentiation of the matrix $N_t Y_t$ followed by a substitution of \eqref{eq:solution-form-problem-2}, \eqref{eq:null-Zt} and \eqref{eq:ivp-nullspace-basis} yields
    \begin{align*}
        \dot{N}_t Y_t + N_t\dot{Y}_t&= 0,
    \end{align*}
    implying
    \begin{align*}
        N_t Y_t = N_0 Y_0 = 0,  \ \ \forall t\in[0,1].
    \end{align*}
    
    From the above, it follows that any square integrable curve $(Z_t)_{t\in[0,1]}\in\mathbb{R}^{n\times n}$ satisfying the constraint $Z_t Y_t= 0$ must be of the form $Z_t = R_t N_t$, for an almost everywhere unique square integrable curve $(R_t)_{t\in[0,1]}\in\mathbb{R}^{n\times (n-m)}$. Finally, since $1\in\mathcal{T}_2=[0,1]$, the curve $(\Phi_t)_{t \in[0,1]}$ must be such that $\Phi_1$ satisfies \eqref{eq:tracer-displacement}. The converse implication of the proposition follows by direct integration.
\end{proof}

 In essence, Equation \eqref{eq:Mat-N0} introduces a matrix $N_0$ whose null space contains the column space of $Y_0$.
If the dynamics of $N_t$ coincide with \eqref{eq:dynamics-Nt}, then the column space of $Y_t$ is ensured to remain in the null space of $N_t$ for all $t\in [0,1]$. As such, Proposition~\ref{prop:admissible-curves-problem-2} suggests a reformulation of Problem~\ref{problem:dual}, replacing condition \eqref{eq:Yt-constraint} on the state transition $\Phi_t$ with an explicit constraint on the gain $K_t$. The advantage of the new formulation lies in the fact that the search for the optimal curve reduces to searching over a lower-dimensional control space, namely, the curve $(R_t)_{t\in[0,1]}$.

\begin{taggedproblem}{\ref{problem:dual}$^\prime$} \label{problem:dual-reformulation} 
    Given $(\Sigma_t)_{t \in \mathcal{T}_1} \in {\rm Sym}^+(n)$ with $\mathcal{T}_1=\{0,1\}$ and $(Y_t)_{t  \in \mathcal{T}_2} \in \mathbb R^{n \times m}$ with $\mathcal{T}_2=[0,1]$, find 
    \begin{align*}
        (K_t^\star, \Phi_t^\star): = \arg \min_{(K_t, \Phi_t)} J_{\rm KE} + \varepsilon J_{\rm A},
    \end{align*}
    for $\varepsilon >0$ and subject to the constraints \eqref{eq:state-transition}, \eqref{eq:covariance-constraint}, \eqref{eq:tracer-displacement}, and \eqref{eq:solution-form-problem-2}. 
\end{taggedproblem}

Before we proceed with deriving the necessary conditions of optimality, we state in passing a theorem on the existence of solutions to Problem~\ref{problem:dual-reformulation}.

\begin{thm}\label{thm:existence-of-solutions-problem-2}
    Suppose that there exists at least one admissible curve for Problem~\ref{problem:dual-reformulation}. Then, there exists an absolute minimizer $(K_t^\star,\Phi_t^\star)_{t\in[0,1]}$ for Problem~\ref{problem:dual-reformulation}.
\end{thm}
%

%

Similar to Theorem \ref{thm:existence-of-solutions-problem-1}, the proof of the above statement consists of verifying that the conditions of \cite[Theorem 11.4.vi]{cesari1983optimization} hold. 
We now proceed to state the necessary conditions of optimality for Problem~\ref{problem:dual-reformulation}.

\begin{prop}
    The absolute minimizer $(K_t^\star ,\Phi_t^\star)_{t \in [0,1]}$ of Problem~\ref{problem:dual}, or equivalently Problem~\ref{problem:dual-reformulation}, is such that
    \begin{subequations}
    \begin{align}\label{eq:neccPhi2}
        \dot{\Phi}_t^\star  &=  K_t^\star  \Phi_t^\star , 
        \end{align}
        with $\Phi_0^\star  = I$ and  $K_t^\star$ coincides with \eqref{eq:solution-form-problem-2} for
        \begin{align}
           \hspace{-20pt}   R_t &=   -(M_t B + P_t^\star (\Phi_t^\star)^\top ) N_t^\top (N_t B N_t^\top)^{-1},  \label{eq:Rt-mat} \\
     \hspace{-10pt}   \text{where } \hspace{5pt}   \dot{P}_t^\star& =  -{(K_t^\star)}^\top (K_t^\star \Phi_t^\star \Sigma_0 +P_t^\star),      
          \label{eq:neccP2} \\ 
 \hspace{-20pt}   \text{and } \hspace{10pt} B &=\big(\Phi_t^\star \Sigma_0 (\Phi_t^\star)^\top +\varepsilon I \big).  \label{eq:Bmatrix} 
        \end{align}
        Moreover, if $\mathcal S\subset\mathrm{GL}(n)$ is the surface defined by
        \begin{align} \label{eq:surfnecc2}
            \mathcal{S}:=\{\Phi\in\mathrm{GL}(n)~|~\Phi\Sigma_0\Phi^\top=\Sigma_1,\,\Phi Y_0= Y_1\},
            \end{align}
      \end{subequations}
        then \eqref{eq:neccPhi2} and \eqref{eq:neccP2} are subject to the mixed boundary conditions $\Phi_1^\star  \in \mathcal S$ and ${P_1^\star} \perp T_{\Phi_1^\star }\mathcal{S}$, respectively.          
\end{prop}

\begin{proof}
   As before, we express the control Hamiltonian of Problem~\ref{problem:dual-reformulation} as
         \begin{align*}
         \mathcal H(K_t, \Phi_t,P_t) &=  \frac{1}{2} \tr\big(K_t \Phi_t \Sigma_0  \Phi^\top_t K_t^\top + \varepsilon K_t K_t^\top) \\
         &+ \tr(P_t^\top K_t \Phi_t\big) ,
    \end{align*}
     for $K_t$ satisfying \eqref{eq:solution-form-problem-2} and for $P_t$ an $n \times n$ Lagrange multiplier matrix that enforces \eqref{eq:state-transition}.
    Equations~\eqref{eq:neccPhi2} and \eqref{eq:neccP2} are obtained respectively from the necessary conditions  $\dot{\Phi}_t^\star  = \partial \mathcal H/\partial P_t $ and $\dot{P}_t^\star  = -\partial \mathcal H/\partial \Phi_t $. 
     Exploiting the form in \eqref{eq:solution-form-problem-2}, we rewrite the previous Hamiltonian as
    \begin{align*}
         \mathcal H(R_t, \Phi_t,P_t) &=  \frac{1}{2} \tr\big((M_t+R_t N_t) B (M_t+R_t N_t) ^\top\big)\\
        &+ \tr \big( P_t^\top \big(M_t+R_t N_t\big)\Phi_t\big),
    \end{align*}
    for $B$ satisfying \eqref{eq:Bmatrix}. Indeed, direct application of Pontryagin's maximum principle, $\partial \mathcal H/ \partial R_t = 0$, results in \eqref{eq:Rt-mat}.    
    Lastly, Equations \eqref{eq:push-lin} and \eqref{eq:Yt-constraint} define the surface in \eqref{eq:surfnecc2} and conditions on $\Phi_1^\star$ and $P_1^\star$ from the proof of Proposition~\ref{prop:proptracer} apply verbatim here, which concludes the proof.
\end{proof}



\section{Examples}\label{sec:examples}

\subsection{Example I}\label{sec:example-1}
\begin{figure}[t]
    \centering
          \begin{subcaptionblock}{\linewidth}
              \centering
      \includegraphics[width=0.8\linewidth]{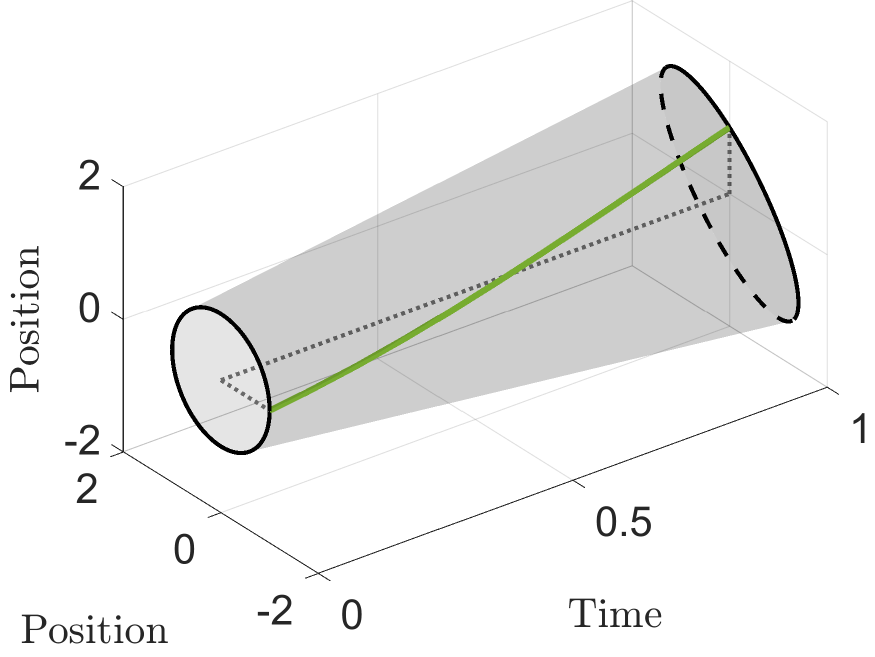}
      \caption{The gray envelope represents the given $(\Sigma_t)_{t \in [0,1]}$ as in \eqref{eq:sigmas_ex1}. The green curve is the resulting trajectory of the tracer between two given endpoints \eqref{eq:tracer0,1}.} \label{fig:Ex1.1}
     \end{subcaptionblock}\\[10pt]
        \begin{subcaptionblock}{\linewidth}
    \centering
    \includegraphics[width=0.8\linewidth]{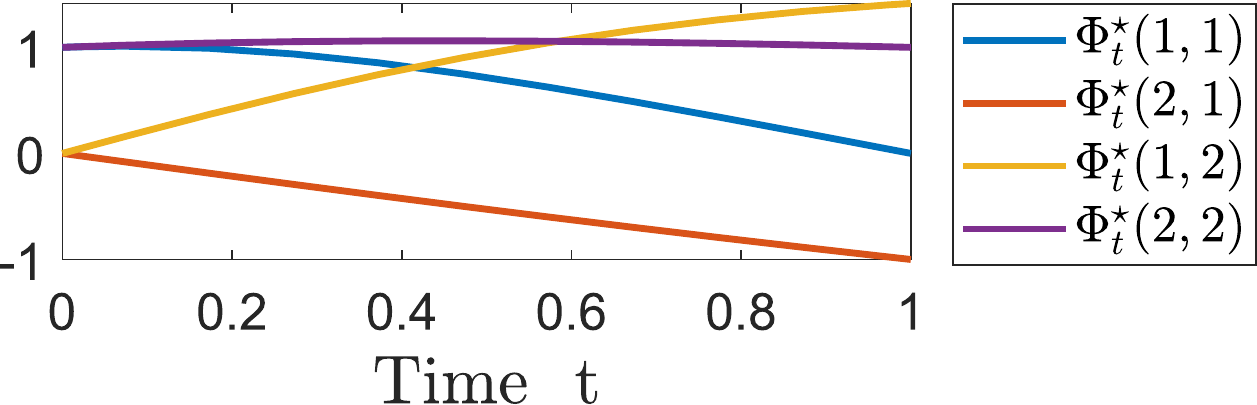}
    \caption{Evolution of the entries of the state transition matrix.} \label{fig:Ex1.2}
      \end{subcaptionblock}
      \caption{Example I}
\end{figure}
We present an academic example for Problem~\ref{problem:knowledge-on-ens-reformulation}. We assume the positions of a single tracer particle at $t=0,t=1$ are given as
\begin{align}\label{eq:tracer0,1}
    Y_0 = \begin{bmatrix}
          -1 \\ 0
    \end{bmatrix}, \ Y_1 = \begin{bmatrix}
        0 \\ 1
    \end{bmatrix}.
\end{align}
Information on the ensemble is also given as
\begin{subequations} \label{eq:sigmas_ex1}
    \begin{align}
    \Sigma_0 =  I&, \  \Sigma_1 = \begin{bmatrix}
         2 &  \sqrt{2} \\
        \sqrt{2} & 2
    \end{bmatrix}, \\
    \Sigma_t = \big((1-t)I+& t \Sigma_1^{\frac{1}{2}}\big)  \big((1-t)I+ t \Sigma_1^{\frac{1}{2}}\big)^\top,\label{eq:McCannflow}
\end{align}
\end{subequations}
for $t \in [0,1]$. The flow $\Sigma_t$ corresponds to the optimal mass transport between $\Sigma_0$ and $\Sigma_1$ (McCann's interpolation \cite{mccann1997convexity,villani2003topics}) with no data on tracer particles. This is depicted in Fig.~\ref{fig:Ex1.1}, where the gray cylinder represents the evolution of an equi-probability level set of the corresponding Gaussian distributions. The availability of data on the terminal disposition of a tracer particle reveals a rotational component of the flow that cannot be gleaned from the McCann flow in \eqref{eq:McCannflow}.

Since $\Phi_1^\star  {(\Phi_1^\star)}^\top = \Sigma_1$, $\Phi_1^\star  = \Sigma_1^{\frac{1}{2}}U_1$, for some $U_1 \in \rm{SO}(2)$, the special orthogonal group of dimension $2$. In the present case, where $n=2$, the matrix $U_1$ is fully determined by the information provided. Therefore, one can solve for $U_1$ satisfying $\Sigma_1^{-\frac{1}{2}} Y_1 = U_1 Y_0$,
for $Y_0,Y_1$ as in \eqref{eq:tracer0,1}. This results in
\begin{align*}
    U_1 = \begin{bmatrix}
       \frac{1}{2} \sqrt{2-\sqrt{2}}& \frac{1}{2} \sqrt{\sqrt{2}+2}\\
       -\frac{1}{2} \sqrt{\sqrt{2}+2}& \frac{1}{2} \sqrt{2-\sqrt{2}}
    \end{bmatrix},
\end{align*}
which corresponds to a rotation by $67.5^\circ$. Then, 
\begin{align}\label{eq:phiex1}
    \Phi_1^\star  = \begin{bmatrix}
        0 & \sqrt{2} \\
        -1 & 1
    \end{bmatrix},
\end{align} and the system \eqref{eq:neccPhi}-\eqref{eq:neccP} can be solved as a two-point boundary value problem between $\Phi_0^\star =I$ and \eqref{eq:phiex1}. 
The transitions $\Phi_t^\star$ shown in Fig.~\ref{fig:Ex1.2} are obtained from the shooting method. \footnote{The source code is available at:\\
\texttt{https://github.com/a-eld/TransportWithTracers-1.}}These state transitions result in the trajectory depicted in green in Fig.~\ref{fig:Ex1.1}, indeed reconciling the information in \eqref{eq:tracer0,1} and \eqref{eq:sigmas_ex1}. 

\subsection{Example II}
In contrast to the previous example, we assume here that the full trajectory of a single tracer is given, as shown in Fig.~\ref{fig:Ex2.1} in green. This is precisely a trajectory between the endpoints:
\begin{align} \label{eq:tracer0-1}
    Y_0 = \begin{bmatrix}
        -\frac{\sqrt{3}}{2} \\ \frac{1}{2}\\[3pt]
    \end{bmatrix},  \ Y_1 = \begin{bmatrix}
        0 \\ -\sqrt{3}
    \end{bmatrix}.
\end{align}
Moreover, we assume the endpoint covariances are
\begin{align*}
    \Sigma_0 = I,  \  \Sigma_1 = 3I.
\end{align*}

The task here is to numerically demonstrate the solution of Problem~\ref{problem:dual-reformulation} in the two-dimensional case. Similar to the previous example, the terminal state transition matrix can be fully determined as
$\Phi_1^\star  = \Sigma_1^{\frac{1}{2}} O_1$, for 
\begin{align*}
    O_1 = \begin{bmatrix}
         -\frac{1}{2} & -\frac{\sqrt{3}}{2} \\[3pt]
 \frac{\sqrt{3}}{2} & -\frac{1}{2} 
    \end{bmatrix} \in \rm{SO}(2),
\end{align*}
which amounts to a rotation of $-120^\circ$. The resulting $\Phi_1^\star$ along with $\Phi_0^\star=I$ are used to solve the two-point boundary value problem associated with system \eqref{eq:neccPhi2}-\eqref{eq:Bmatrix} for $\varepsilon =1$  and
\begin{align*}
    N_0 = \begin{bmatrix}
        \frac{1}{2} & \frac{\sqrt{3}}{2}
    \end{bmatrix}.
\end{align*} The obtained entries of $\Phi_t^\star$ are presented in Fig.~\ref{fig:Ex2.2}.\footnote{The state transitions are obtained using MATLAB's solver \texttt{bvp5c}. The associated source code can be found at:\\ \texttt{https://github.com/a-eld/TransportWithTracers-2.}} The gray envelope shown in Fig.~\ref{fig:Ex2.1} delineates the equi-probability level sets of the corresponding Gaussian distributions.
\begin{figure}[t] 
    \centering
          \begin{subcaptionblock}{\linewidth}
              \centering
        \includegraphics[width=0.8\linewidth]{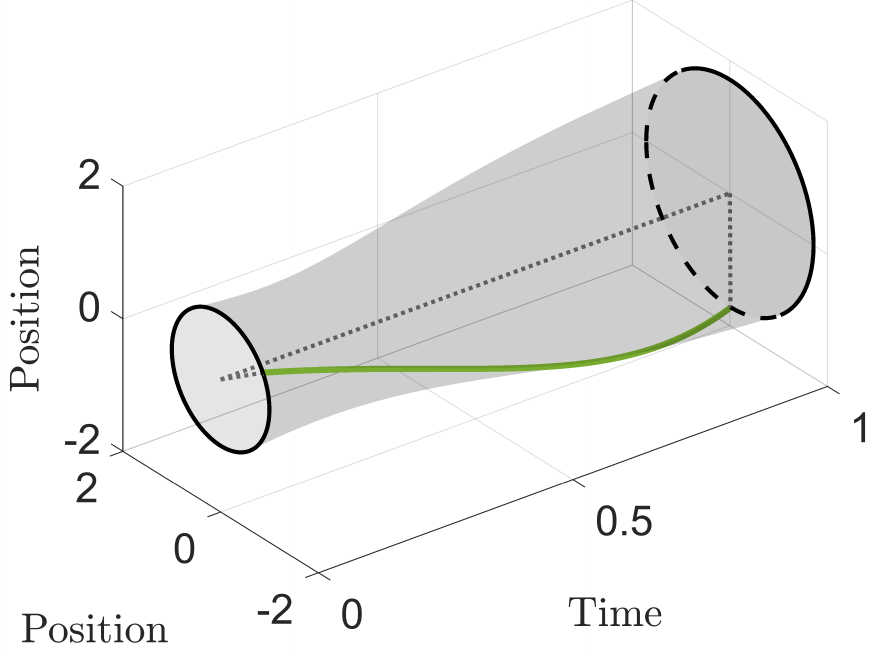}
       \caption{The green curve is the given trajectory of the tracer between two endpoints \eqref{eq:tracer0-1} while the gray envelope depicts the resulting covariances.} \label{fig:Ex2.1}
     \end{subcaptionblock}\\[10pt]
     \begin{subcaptionblock}{\linewidth}
    \centering
    \includegraphics[width=0.8\linewidth]{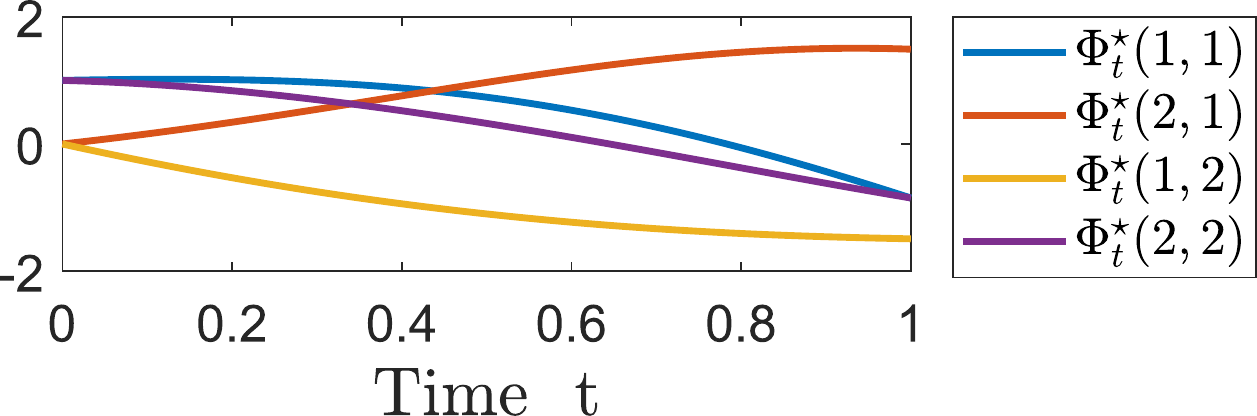}
    \caption{Evolution of the entries of the state transition matrix.}\label{fig:Ex2.2}
          \end{subcaptionblock}
     \caption{Example II}
\end{figure}

\section{Concluding remarks}

The theme of the present work brings up a hitherto unstudied angle in optimal transport problems, where additional data now reveal the dynamics of internal degrees of freedom in the flow. Paradigmatically, in either fluid flows or the flows of a flock of birds, tagging specific individuals (tracer particles or birds, respectively) may provide information about rotational components of the collective dynamics that are not reflected in the aggregate distributions.

This work explores the theme in an elementary setting of linear dynamics and Gaussian distributions. 
From a control perspective, we envision stochastic formulations, where measurement noise corrupts the available data, that mirror the deterministic setting presented herein.
From an applications standpoint, tracer information is naturally available in several fields. For instance, in computer vision, tracers may correspond to landmarks or facial features, and, in fluid flows, tracers are often injected for precisely the same reasons as postulated herein. Moreover, it is of interest to explore how the conceptualization presented can be utilized in related developments, such as \cite{halder2014geodesic, caluya2020finite, elamvazhuthi2023dynamical, mei2024flow}.

\balance
\bibliographystyle{ieeetr}
\bibliography{main.bib}
\end{document}